\documentclass[11pt]{amsart} 
\usepackage{amsmath,amssymb,rawfonts}
\usepackage{tikz}
\usepackage{pbsi}
\usepackage[T1]{fontenc}
\usetikzlibrary{matrix,arrows}

\newtheorem{theorem}{Theorem}[section]

\newtheorem{defi}[theorem]{Definition}
\newtheorem{prop}[theorem]{Proposition}
\newtheorem{coro}[theorem]{Corollary}

\begin{document} 
\title{Interor and $\mathfrak h$ Operators on the Category of Locales}
\author{Joaqu\'in Luna-Torres }
\thanks{Programa de Matem\'aticas, Universidad Distrital Francisco Jos\'e de Caldas,  Bogot\'a D. C., Colombia (retired professor)}

\email{jlunator@fundacionhaiko.org}
\begin{abstract}
We present the concept of interior operator $I$ on the category  $\mathbf{Loc}$ of locales and then we construct a topological category \linebreak  $\big(\mathbf{I\text{-}Loc},\ U\big)$, where $U:\mathbf{I\text{-}Loc}\rightarrow \mathbf{Loc}$ is a forgetful functor;   and we also introduce the notion of $\mathfrak h$ operator on the category  $\mathbf{Loc}$  and discuss some of their properties for constructing the topological  category $\big(\mathbf{\mathfrak h\text{-}Loc},\ U\big)$ associated to the  forgetful functor $U:\mathbf{\mathfrak h\text{-}Loc}\rightarrow \mathbf{Loc}$.

\end{abstract}
\maketitle 
\baselineskip=1.7\baselineskip
\section*{0. Introduction}
Kuratowski operators (closure, interior, exterior, boundary and others)
have been used intensively in General Topology (\cite{Du}, \cite{K1}, \cite{K2}). For a topological space it is well-known that, for example, the associated closure and interior operators provide equivalent descriptions of the topology; but this is not always true in other categories, consequently it makes sense to define and study separately these operators. In this context, we study  an interior operator $I$ on the the coframe $\mathcal{S}_\text{{\bsifamily{l}}}(L)$ of sublocales of every  object  $L$ in the category $\mathbf{Loc}$.

On the other hand, a new topological operador $\mathfrak h$ was introduced by M. Suarez \cite{MSM}  in order to complete a Boolean algebra with all topological operators  in General Topology. Following his ideas, we study an operator $\mathfrak h$ on the  colection $\mathcal{S}_\text{{\bsifamily{l}}}^{c}(L)$ of all complemented sublocales of  every  object  $L$ in the category $\mathbf{Loc}$.
 
The paper is organized as follows, we begin presenting, in section 1,  the basic concepts of Heyting algebras, Frames, locales, sublocales, images and preimages of sublocales for the morphisms of $\mathbf{Loc}$ and the notions of closed and open sublocales; these notions can be found in  Picado and Pultr \cite{PP} and A. L. Suarez \cite{ALS}, In section 2, we present the concept of interior operator $I$ on the category  $\mathbf{Loc}$ and then we construct a topological category  $\big(\mathbf{I\text{-}Loc},\ U\big)$, where $U:\mathbf{I\text{-}Loc}\rightarrow \mathbf{Loc}$ is a forgetful functor. 
Finally in section 3 we present the notion of $\mathfrak h$ operator on the category  $\mathbf{Loc}$  and discuss some of their properties for constructing the topological  category $\big(\mathbf{\mathfrak h\text{-}Loc},\ U\big)$ associated to the  forgetful functor $U:\mathbf{\mathfrak h\text{-}Loc}\rightarrow \mathbf{Loc}$.

\section{Preliminaries}
For a comprehensive account on the the categories of frames and locales we refer to Picado and Pultr \cite{PP} and A. L. Suarez \cite{ALS}, from whom we take the following useful facts.
\subsection{Heyting algebras}
A bounded lattice $L$ is called a Heyting algebra if there is a binary
operation $x \rightarrow y$ (the Heyting operation) such that for all $a, b, c$ in $L$,
\[
a \land b \leqslant c\,\ \text{iff}\,\  a \leqslant b \rightarrow c.
\]
Thus  for every $b \in L$ the mapping $b \rightarrow (-) : L \rightarrow L$ is a right adjoint to $(-) \land b : L \rightarrow L$ and hence, if it exists, is uniquely determined.

In a complete Heyting algebra we have $(\bigvee A) \land b = \bigvee_{a\in A}(a \land b)$ for any $A\subseteq L$, $b\rightarrow (\bigvee A) = \bigvee_{a\in A}(b \rightarrow a)$, and $(\bigvee A) \rightarrow b = \bigvee{a\in A}(a \rightarrow b)$.
\subsection{Frames}
A {\it frame} is a complete lattice $L$ satisfying the distributive law
\[
\big(\bigvee A\big)\land b =\bigvee \{ a\land b\mid a\in A\}
\]
for all $A \subseteq L$ and $b\in L$ (hence a complete Heyting algebra); a {\it frame homomorphism} preserves all joins and all finite meets.

The lattice $\Omega(X)$ of all open subsets of a topological space $X$ is an example of a frame, and if $f: X \rightarrow Y$ is continuous we obtain a frame homomorphism 
$\Omega( f ): \Omega(Y )\rightarrow \Omega (X)$ by setting $\Omega(f)(U) = f^{-1}[U]$. Thus we have a contravariant functor  $\Omega : \mathbf{Top} \rightarrow \mathbf{Frm}^{op} $ from the category of topological spaces into that of frames. \subsection{Locales}
The adjunction $\Omega : \mathbf{Top} \rightarrow \mathbf{Frm}^{op} $, \,\
$\mathbf{pt}:\mathbf{Frm}^{op}\rightarrow \mathbf{Top}$ with $\Omega  \dashv \mathbf{pt}$, connects the categories of frames with that of topological spaces. The functor $\Omega$  assigns to each space its lattice of opens, and $\mathbf{pt}$ assigns to a frame $L$ the collection of the frame maps $f : L \rightarrow 2$, topologized by setting the opens to be exactly the sets of the form 
$ \{f: L \rightarrow 2 \mid f(a) = 1\}$ for some $a\in L$.

 A frame $L$ is {\it spatial} if for $a, b \in L$ whenever $a\nleqslant b$ there is some frame map $f: L \rightarrow 2$ such that $f(a) = 1\ne f(b)$. Spatial frames are exactly those of the form $\Omega(X)$ for some space $X$.

 A space is {\it sober} if every irreducible closed set is the closure of a unique point. Sober spaces are exactly those of the form $\mathbf{pt}(L)$ for some frame $L$.
 
The adjunction $\Omega\dashv \mathbf{pt}$ restricts to a dual equivalence of categories between spatial frames and sober spaces
 
The category of sober spaces is a full reflective subcategory of $\mathbf{Top}$. For each space $X$ we have a sobrification map 
$N : X\rightarrow  \mathbf{pt}(\Omega(X))$ mapping each point $x\in  X$ to the map
$f_x :(X) \rightarrow 2$ defined as $f(U) = 1$ if and only if $x\in  U$.

The category of {\it spatial frames} is a full reflective subcategory of $\mathbf{Frm}$. For each frame we have a spatialization map $\phi : L \rightarrow \Omega(\mathbf{pt}(L))$ which sends each $a\in L$ to $\{f : L \rightarrow 2 \mid f(a) = 1\}$.

This justifies to view the dual category $\mathbf{Loc} =\mathbf{Frm}^{op}$ as an extended category of spaces; one speaks of the category of {\it locales}.

Maps in the category of locales have a concrete description: they can be characterized as the right adjoints of frame maps (since frame maps preserve all joins, they always have right adjoints).
\subsubsection{\bf{Sublocales}}
A {\it sublocale} of a locale $L$ is a subset $S\subseteq L$ such that it is closed under arbitrary meets, and such that $s\in S$ implies $x\rightarrow s \in S$ for every $x \in L$. This is equivalent to $S\subseteq L$ being a locale in the inherited order, and the subset inclusion being a map in $\mathbf{Loc}$.

Sublocales of $L$ are closed under arbitrary intersections, and so the collection 
 $\mathcal{S}_\text{{\bsifamily{l}}}(L)$ of all sublocales of $L$, ordered under set inclusion, is a complete lattice.  The join of sublocales is (of course) not the union, but we have a very simple formula 
   $\bigvee_{i} S_{i} = \{\bigvee M \mid M \subseteq\bigcup_{i} S_{i}\}$.
   
 In the coframe $\mathcal{S}_\text{{\bsifamily{l}}}(L)$ the bottom element is the sublocale $\{1\}$ and the top element is $L$.

\subsubsection{\bf{ Images and Preimages of sublocales}}
Let $f: L\rightarrow M$  be a localic map and if $S \subseteq L$ is a sublocale then the standard set-theoretical image $f [S]$ is also a sublocale The set-theoretic preimage $f^{-1}[T]$ of a sublocale $T\subseteq M$ is not necessarily a sublocale of $L$. To obtain a concept of a preimage suitable for our purposes we will, first, make the following observation: ``Let $A\subseteq L$ be a subset closed under meets. Then $\{1\} \subseteq A$ and if $S_i \subseteq A$ for $i \in J$ then $\bigwedge_{i\in J} S_i\subseteq A$''.
Consequently there exists the largest sublocale contained in  $A$. It will be denoted by $A_{sloc}$. 

The set-theoretic preimage $f^{-1}[T]$ of a sublocale $T$ is closed under meets \big(indeed, $f(1) = 1$, and if $x_i \in f^{-1}[T])$ then $f(x_i) \in T$,  and hence
$ f(\bigwedge_{i\in J} x_i)=\bigwedge_{i\in J} f(x_i)$ belongs to $T$ and $\bigwedge_{i\in J} x_i\in f^{-1}[T]$ \big) and we have the sublocale 
$f_{-1}[T]:= f^{-1}[T]_{sloc}$. It will be referred to as {\it the preimage} of $T$, and we shall sat that $f_{-1}[-]$ is {\it the preimage function} of $f$.
 
For every localic map  $f: L \rightarrow M$, the preimage function $f_{-1}[-] $ 
is a right Galois adjoint of the image function $f [-] :\mathcal{S}_\text{{\bsifamily{l}}}(L)\rightarrow \mathcal{S}_\text{{\bsifamily{l}}}(M)$.

\subsubsection{\bf{ Closed and Open sublocales}}\label{open}
Embedded in  $\mathcal{S}_\text{{\bsifamily{l}}}(L)$ we have the coframe of {\it closed sublocales} which is isomorphic to $L^{op}$.
The closed sublocale $\mathfrak c(a) \subseteq L$ is defined to be $\uparrow a$ for $a \in L$.

Embedded in  $\mathcal{S}_\text{{\bsifamily{l}}}(L)$ we also have the frame of open sublocales which is isomorphic to $L$. The open sublocale is defined to be 
$\{a \rightarrow x \mid x \in L\}$ for $a \in L$.

The sublocales $\mathfrak o(a)$  and $\mathfrak c(a)$  are complements of one another in the coframe  $\mathcal{S}_\text{{\bsifamily{l}}}(L)$ for any element $a\in L$.
Furthermore, open and closed sublocales generate the coframe  $\mathcal{S}_\text{{\bsifamily{l}}}(L)$ in the sense that for each \linebreak $S \in \mathcal{S}_\text{{\bsifamily{l}}}(L)$ we have $S = \bigcap\{\mathfrak o(x) \cup \mathfrak c(y) \mid  S \subseteq \mathfrak o(x) \cup \mathfrak c(y)\}$.

A pseudocomplement of an element $a$ in a meet-semilattice $L$ 
 with $0$ is the largest element $b$ such that $b\land a = 0$, if it exists. It is usually denoted by $\neg a$.  Recall  that in a Heyting algebra $H$ the pseudocomplement can be expressed as $\neg x= x\rightarrow  0$.
 
\section{Interior Operators}
We shall be conserned in this section with a version on locales of the interior operator studied in \cite{LO}.

Before stating the next definition, we need to observe that since for  localic maps $f: L \rightarrow M$ and $g:M\rightarrow N$:
\begin{itemize}
\item the preimage function $f_{-1}[-] $ 
is a right Galois adjoint of the image function $f [-] :\mathcal{S}_\text{{\bsifamily{l}}}(L)\rightarrow \mathcal{S}_\text{{\bsifamily{l}}}(M)$;
\item $ g [-]\circ f [-]=(g\circ f) [-]$.
\end{itemize}
Therefore $g_{-1}[-]\circ f_{-1}[-]= (g\circ f)_{-1}[-]$ because given two adjunctions the composite functors yield an adjunction.
 \begin{defi}
 An interior operator $I$ of the category  $\mathbf{Loc}$ is given by a family $I =(i_{\text{\tiny{$L$}}})_{\text{$L\in \mathbf{Loc}$}}$ of maps $i_{\text{\tiny{$L$}}}:\mathcal{S}_\text{{\bsifamily{l}}}(L)\rightarrow \mathcal{S}_\text{{\bsifamily{l}}}(L)$ such that
 \begin{itemize}
 \item[($I_1)$] $\left(\text{Contraction}\right)$\,\  $i_{\text{\tiny{$L$}}}(S)\subseteq S$ for all $S \in \mathcal{S}_\text{{\bsifamily{l}}}(L)$;
 \item[($I_2)$] $\left(\text{Monotonicity}\right)$\,\  If $S\subseteq T$ in $\mathcal{S}_\text{{\bsifamily{l}}}(L)$, then $i_{\text{\tiny{$L$}}}(S)\subseteq i_{\text{\tiny{$L$}}}(T)$
 \item[($I_3)$] $\left(\text{Upper bound}\right)$\,\  $i_{\text{\tiny{$L$}}}(L)=L$.
 \end{itemize}
 \end{defi}
 \begin{defi}
 An $I$-space is a pair $(L,i_{\text{\tiny{$L$}}})$ where $L$ is an object of $\mathbf{Loc}$ and $i_{\text{\tiny{$L$}}}$ is an interior operator on $L$.
 \end{defi}
 \begin{defi}
 A morphism $f:L\rightarrow M$ of  $\mathbf{Loc}$ is said to be $I$-continuous if 
 \begin{equation}\label{conti}
 f_{-1}\left[ i_{\text{\tiny{$M$}}}(T)\right]\subseteq i_{\text{\tiny{$L$}}}\left( f_{-1}[T]\right)
 \end{equation}
 for all $T\in \mathcal{S}_\text{{\bsifamily{l}}}(M)$. Where $f_{-1}[-]$ is the preimage  of $f[-]$.
 \end{defi} 
 
 \begin{prop}
 Let $f:L\rightarrow M$ and $g:M\rightarrow N$ be two $I$-continuous morphisms  of  $\mathbf{Loc}$ then $g\centerdot f$ is an $I$-continuous morphism of  $\mathbf{Loc}$.
 \end{prop}
 \begin{proof}
 Since $g:M\rightarrow N$ is $I$-continuous, we have  $$g_{-1}\big[ i_{\text{\tiny{$N$}}}(S)\big]\subseteq i_{\text{\tiny{$M$}}}\big( g_{-1}[S]\big)$$
  for all $S\in \mathcal{S}_\text{{\bsifamily{l}}}(N)$, it fallows that
 $$f_{-1}\Big[g_{-1}\big[( i_{\text{\tiny{$N$}}}(S)\big]\Big]\subseteq f_{-1}\Big[ i_{\text{\tiny{$M$}}}\big( g_{-1}[S]\big)\Big];$$
  now,  by the $I$-continuity of $f$,$$ f_{-1}\Big[ i_{\text{\tiny{$M$}}} \big( g_{-1}[S]\big)\Big]\subseteq i_{\text{\tiny{$L$}}}\Big( f_{-1}\big[g_{-1}[S]\big]\Big),$$ 
therefore $$f_{-1}\Big[g_{-1}\big[ i_{\text{\tiny{$N$}}}(S)\big]\Big]\subseteq i_{\text{\tiny{$L$}}}\Big( f_{-1}\big[g_{-1}[S]\big]\Big),$$
  that is to say 
  $$(g\centerdot f)_{-1}\big[ i_{\text{\tiny{$N$}}}(S)\Big]\subseteq i_{\text{\tiny{$L$}}}\Big( (g\centerdot f)_{-1}[S]\Big)$$
\end{proof}
  As a consequence we obtain
 \begin{defi}
 The category $\mathbf{I\text{-}Loc}$ of $I$-spaces comprises the following data:
 \begin{enumerate}
 \item {\bf Objects}: Pairs $(L,i_{\text{\tiny{$L$}}})$ where $L$ is an object of $\mathbf{Loc}$ and $i_{\text{\tiny{$L$}}}$ is an interior operator on $L$.
\item {\bf Morphisms}: Morphisms of $\mathbf{Loc}$ which are $I$-continuous.
 \end{enumerate}
 \end{defi}

 \subsection{The lattice structure of all interior operators}
 For the category $\mathbf{Loc}$ we consider the collection
 \[
 Int(\mathbf{Loc})
 \]
 of all  interior operators on $\mathbf{Loc}$. It is ordered by
\[
 I\leqslant J \Leftrightarrow i_{\text{\tiny{$L$}}}(S)\subseteq j_{\text{\tiny{$L$}}}(S), \,\,\ \text{for all $S\in \mathcal{S}_\text{{\bsifamily{l}}}$ and all $L$  object of $\mathbf{Loc}$}. 
\]

 This way $Int(\mathbf{Loc})$ inherits a lattice structure from $\mathcal{S}_\text{{\bsifamily{l}}}$:
 
 \begin{prop}
 Every family $(I_{\text{\tiny{$\lambda$}}})_{\text{\tiny{$\lambda\in \Lambda$}}}$ in $Int(\mathbf{Loc})$ has a join $\bigvee\limits_{\text{\tiny{$\lambda\in \Lambda $}}}I_{\text{\tiny{$\lambda $}}}$ and a meet $\bigwedge\limits_{\text{\tiny{$\lambda\in \Lambda $}}}I_{\text{\tiny{$\lambda $}}}$ in $Int(\mathbf{Loc})$. The discrete interior operator
 \[ I_{\text{\tiny{$D$}}}=({i_{\text{\tiny{$D$}}}}_{\text{\tiny{$L$}}})_{\text{$L\in \mathbf{Loc}$}}\,\,\ \text{with}\,\,\ {i_{\text{\tiny{$D$}}}}_{\text{\tiny{$L$}}}(S)=S\,\,\ \text{for all}\,\ S\in \mathcal{S}_\text{{\bsifamily{l}}}
 \]
  is the largest element in $Int(\mathbf{Loc})$, and the trivial interior operator
  \[ 
I_{\text{\tiny{$T$}}}=({i_{\text{\tiny{$T$}}}}_\text{\tiny{$L$}})_{\text{$L\in \mathbf{Loc}$}}\,\,\ \text{with}\,\,\ {i_{\text{\tiny{$T$}}}}_{\text{\tiny{$L$}}}(S)=
  \begin{cases}
\{1\}& \text{for all}\,\ S\in \mathcal{S}_\text{{\bsifamily{l}}},\,\ S\ne L\\
L&\text {if}\,\ S=L
\end{cases}
 \]
is the least one.
\end{prop}

\begin{proof}
For $\Lambda\ne\emptyset$, let $\widehat{I}=\bigvee\limits_{\text{\tiny{$\lambda\in\Lambda $}}}I_{\text{\tiny{$\lambda $}}}$, then 
 \[
 \widehat{i_{\text{\tiny{$L$}}}}=\bigvee\limits_{\text{\tiny{$\lambda\in \Lambda$}}} {i_{\text{\tiny{$\lambda $}}}}_{\text{\tiny{$L$}}},
 \]
 for all  $L$ object of $\mathbf{Loc}$, satisfies
\begin{itemize}
 \item $ \widehat{i_{\text{\tiny{$L$}}}}(S)\subseteq S$,  because ${i_{\text{\tiny{$\lambda $}}}}_{\text{\tiny{$L$}}}(S)\subseteq S$ for all $S\in \mathcal{S}_\text{{\bsifamily{l}}}$ and for all $\lambda \in \Lambda$.
\item If $S\leqslant T$ in $\mathcal{S}_\text{{\bsifamily{l}}}$ then ${i_{\text{\tiny{$\lambda $}}}}_{\text{\tiny{$L$}}}(S)\subseteq {i_{\text{\tiny{$\lambda $}}}}_{\text{\tiny{$L$}}}(T)$  for all $S\in \mathcal{S}_\text{{\bsifamily{l}}}$ and for all $\lambda \in \Lambda$, therefore $ \widehat{i_{\text{\tiny{$S$}}}}(S)\subseteq \widehat{i_{\text{\tiny{$L$}}}}(T)$.
 \item Since ${i_{\text{\tiny{$\lambda $}}}}_{\text{\tiny{$L$}}}(L)=L $  for all $\lambda \in \Lambda$, we have that  $ \widehat{i_{\text{\tiny{$L$}}}}(L)=L$.
 \end{itemize}
Similarly  $\bigwedge\limits_{\text{\tiny{$\lambda\in \Lambda $}}}I_{\text{\tiny{$\lambda $}}}$,\,\  $ I_{\text{\tiny{$D$}}}$ and $I_{\text{\tiny{$T$}}}$ are interior operators.
 \end{proof}
 \begin{coro}\label{complete}
  For  every  object $L$ of $\mathbf{Loc}$
  \[
  Int(L) = \{i_{\text{\tiny{$L$}}}\mid i_{\text{\tiny{$L$}}}\,\ \text{ is an interior operator on}\,\ L\}
  \]
  is a complete lattice.
 \end{coro}

 \subsection{Initial interior operators}
Let  $\mathbf{I\text{-}Loc}$ be the ctegory of $I$-spaces. Let  $(M,i_{\text{\tiny{$M$}}})$ be an object of  $\mathbf{I\text{-}Loc}$ and let $L$ be an object of $\mathbf{Loc}$. For each morphism  $f:L\rightarrow M$ in $\mathbf{Loc}$ we define on $L$ the operotor 
\begin{equation} \label{initial}
i_{\text{\tiny{$L_{f}$}}}:=f_{-1}\centerdot i_{\text{\tiny{$M$}}}\centerdot f_{*}.
\end{equation}
\begin{prop}\label{ini-cont}
The operator (\ref{initial}) is an interior operator on $L$ for which the morphism $f$ is $I$-continuous.
\end{prop}
\begin{proof}\
\begin{enumerate}
\item[($I_1)$] $\left(\text{Contraction}\right)$\,\  $i_{\text{\tiny{$L_{f}$}}}(S)= f_{-1}\centerdot i_{\text{\tiny{$M$}}}\centerdot f_{*}[S]\subseteq f_{-1}\centerdot  f_{*}[S]\subseteq S$ for all $S\in \mathcal{S}_\text{{\bsifamily{l}}}$;
 \item[($I_2)$] $\left(\text{Monotonicity}\right)$\,\   $S\subseteq T$ in $\mathcal{S}_\text{{\bsifamily{l}}}$, implies $f_{*}[S]\subseteq f_{*}[T]$, then $i_{\text{\tiny{$M$}}}\centerdot f_{*}[S]\subseteq i_{\text{\tiny{$M$}}}\centerdot f_{*}[T]$, consequently  $ f_{-1}\centerdot i_{\text{\tiny{$M$}}}\centerdot f_{*}[S]\subseteq f_{-1}\centerdot i_{\text{\tiny{$M$}}}\centerdot f_{*}[T]$;
 \item[($I_3)$] $\left(\text{Upper bound}\right)$\,\  $i_{\text{\tiny{$L_{f}$}}}(L)=f_{-1}\centerdot i_{\text{\tiny{$M$}}}\centerdot f_{*}[L]=L$.
\end{enumerate}
Finally, 
\begin{align*}
f_{-1}\big(i_{\text{\tiny{$M$}}}(T)\big)&\subseteq f_{-1}\big(i_{\text{\tiny{$M$}}}\centerdot f_{*}\centerdot f^{-1}(T)\big)=(f_{-1}\centerdot i_{\text{\tiny{$M$}}}\centerdot f_{*})\big(f^{-1}(T)\big)\\
&= i_{\text{\tiny{$L_{f}$}}}\big(f^{-1}(T)\big),
\end{align*}
 for all $T\in  \mathcal{S}_\text{{\bsifamily{l}}}$.
\end{proof}
It is clear that $ i_{\text{\tiny{$L_{f}$}}}$ is the coarsest interior operator on $L$ for which the morphism $f$ is $I$-continuous; more precisaly
\begin{prop}\label{unique}
Let $(L,i_{\text{\tiny{$L$}}})$ and $(M,i_{\text{\tiny{$M$}}})$ be objects of $\mathbf{I\text{-}Loc}$, and let $N$ be an object of  $\mathbf{Loc}$. For each morphism  $g:N\rightarrow L$ in  $\mathbf{Loc}$ and for\linebreak  $f:(L,i_{\text{\tiny{$L_{f}$}}})\rightarrow (M,i_{\text{\tiny{$N$}}})$ an $I$-continuous morphism, $g$  is $I$-continuous if and only if $f\centerdot g$ is $I$-continuous.
\end{prop}
\begin{proof}
Suppose that $g\centerdot f$ is $I$-continuous, i. e.
$$(f\centerdot g)_{-1}\big(i_{\text{\tiny{$M$}}}(T)\big)\subseteq i_{\text{\tiny{$N$}}}\big( (f\centerdot g)_{-1}(T) \big)$$
 for all $T\in \mathbf{S(N)}$. Then, for all $S\in \mathcal{S}_\text{{\bsifamily{l}}}$, we have
\begin{align*}
 g_{-1}\big(i_{\text{\tiny{$L_{f}$}}}(S)\big)&=g_{-1}\big(f_{-1}\centerdot i_{\text{\tiny{$M$}}}\centerdot f_{*}(S)\big)=(f\centerdot g)_{-1}\big( i_{\text{\tiny{$M$}}}( f_{*}(S)) \big)\\
 &\subseteq i_{\text{\tiny{$N$}}}\big( (f\centerdot g)_{-1}(f_{*}(S) ) \big)=i_{\text{\tiny{$N$}}}\big( g_{-1}\centerdot f_{-1}\centerdot  f_{*} (S)\big)\\
 &\subseteq i_{\text{\tiny{$N$}}}\big( g_{-1}(S)\big),\\
\end{align*}
i.e.  $g$  is $I$-continuous.
\end{proof}
As a consequence of corollary(\ref{complete}), proposition(\ref{ini-cont}) and proposition (\ref{unique}) (cf. \cite{AHS} or \cite{JM}), we obtain 
 \begin{theorem}
The forgetful functor $U:\mathbf{I\text{-}Loc}\rightarrow \mathbf{Loc}$ is topological, i.e. the concrete category $\big(\mathbf{I\text{-}Loc},\ U\big)$ is topological.
 \end{theorem}
 \subsection{Open subobjects}
 We introduce a  notion of open subobjects different from the one alluded in \ref{open}.
 
 \begin{defi}
 An sublocale $S$ of a locale $L$ is called $I$-open \big(in $L$\big) if it is isomorphic to its $I$-interior, that is: if $i_{\text{\tiny{$L$}}}(S)= S$.
 \end{defi}
 The $I$-continuity condition (\ref{conti}) implies that $I$-openness is preserve by inverse images:
 \begin{prop}
 Let $f:L\rightarrow M$ be a morphism in $\mathbf{Loc}$. If $T$ is $I$-open in $M$, then $f_{-1}(T)$ is $I$-open in $L$.
 \end{prop}
 \begin{proof}
 If $T= i_{\text{\tiny{$M$}}}(T)$ then $f_{-1}(T)=f_{-1}\big(i_{\text{\tiny{$M$}}}(T)\big)\subseteq i_{\text{\tiny{$L$}}}\big(f_{-1}(T)\big)$, so \linebreak  $i_{\text{\tiny{$L$}}}\big(f_{-1}(T)\big)=f_{-1}(T)$.
 \end{proof}

 \section{$\mathfrak h$ Operators}
 
In this section we shall be conserned with a weak categorical version of a topological function studied by M, Suarez M. in \cite{MSM}. For that purpose we will use the colection $\mathcal{S}_\text{{\bsifamily{l}}}^{c}(L)$ of all complemented sublocales of a locale $L$ (See P, T. Johnston \cite{PJ1}, for example).
 \begin{defi}
 An $\mathfrak h$ operator of the category  $\mathbf{Loc}$ is given by a family $\mathfrak h =(h_{\text{\tiny{$L$}}})_{\text{$L\in \mathbf{Loc}$}}$ of maps $h_{\text{\tiny{$L$}}}:\mathcal{S}_\text{{\bsifamily{l}}}^{c}(L)\rightarrow \mathcal{S}_\text{{\bsifamily{l}}}^{c}(L)$ such that
\begin{itemize}
\item [($h_1$)] $S\cap h_{\text{\tiny{$L$}}}(S)\subseteq S$, for all $S \in\mathcal{S}_\text{{\bsifamily{l}}}^{c}(L)$;
\item [($h_2$)] If $S\subseteq T$ then $S\cap h_{\text{\tiny{$L$}}}(S)\subseteq T\cap h_{\text{\tiny{$L$}}}(T)$,  for all $S,T \in\mathcal{S}_\text{{\bsifamily{l}}}^{c}(L)$;
\item [($h_3$)] $  h_{\text{\tiny{$L$}}}(L)=L$. 
 \end{itemize}
 \end{defi}
  
 \begin{defi}
 An $\mathfrak h$-space is a pair $(L,h_{\text{\tiny{$L$}}})$ where $L$ is an object of $\mathbf{Loc}$ and $h_{\text{\tiny{$L$}}}$ is an $\mathfrak h$  operator on $L$.
 \end{defi}

 \begin{defi}
 A morphism $f:L\rightarrow M$ of  $\mathbf{Loc}$ is said to be $\mathfrak h$-continuous if 
 \begin{equation}\label{h-conti}
 f_{-1}\left[T\cap h_{\text{\tiny{$M$}}}(T)\right]\subseteq f_{-1}[T]\cap h_{\text{\tiny{$L$}}}\left( f_{-1}[T]\right)
 \end{equation}
 for all $T\in \mathcal{S}_\text{{\bsifamily{l}}}^{c}(M)$. Where $f_{-1}[-]$ is the inverse image of $f[-]$.
 \end{defi}
 
 \begin{prop}
 Let $f:L\rightarrow M$ and $g:M\rightarrow N$ be two  $\mathfrak h$-continuous morphisms  of  $\mathbf{Loc}$ then $g\centerdot f$ is an  $\mathfrak h$-continuous morphism of  $\mathbf{Loc}$.
 \end{prop}
  \begin{proof}
 Since $g:M\rightarrow N$ is $I$-continuous, we have  

$$
 g_{-1}\left[V\cap h_{\text{\tiny{$N$}}}(V)\right]\subseteq g_{-1}[V]\cap h_{\text{\tiny{$M$}}}\left( g_{-1}[V]\right)
 $$
 for all $V\in \mathcal{S}_\text{{\bsifamily{l}}}^{c}(N)$, it fallows that
$$
f_{-1}\big[ g_{-1}\left[V\cap h_{\text{\tiny{$N$}}}(V)\right]\big]\subseteq f_{-1}\big[g_{-1}[V]\cap h_{\text{\tiny{$M$}}}\left( g_{-1}[V]\right)\big]
 $$
  now,  by the $\mathfrak h$-continuity of $f$,
$$ f_{-1}\big[g_{-1}[V]\cap h_{\text{\tiny{$M$}}}(g_{-1}[V])\big]\subseteq f_{-1}\big[g_{-1}[V]\big]\cap h_{\text{\tiny{$L$}}}\left( f_{-1}\big[g_{-1}[V]\big]\right)$$
therefore $$(g\centerdot f)_{-1}\big[V\cap h_{\text{\tiny{$N$}}}(V)\big]\subseteq (g\centerdot f)_{-1}\cap h_{\text{\tiny{$L$}}}\big((g\centerdot f)_{-1}[V]  \big).$$
This complete the proof.   
\end{proof}
  
As a consequence we obtain
\begin{defi}
 The category $\mathbf{\mathfrak h\text{-}Loc}$ of $\mathfrak h$-spaces comprises the following data:
 \begin{enumerate}
 \item {\bf Objects}: Pairs $(L,h_{\text{\tiny{$L$}}})$ where $L$ is an object of $\mathbf{Loc}$ and $h_{\text{\tiny{$L$}}}$ is an $\mathfrak h$-operator on $L$.
\item {\bf Morphisms}: Morphisms of $\mathbf{Loc}$ which are $\mathfrak h$-continuous.
 \end{enumerate}
 \end{defi}
 
  \subsection{The lattice structure of all $\mathfrak h$ operators}
 For the category $\mathbf{Loc}$ we consider the collection
 \[
 \mathfrak h(\mathbf{Loc})
 \]
 of all  $\mathfrak h$ operators on $\mathbf{Loc}$. It is ordered by
\[
 \mathfrak h\leqslant \mathfrak{h^{'}} \Leftrightarrow h_{\text{\tiny{$L$}}}(S)\subseteq h^{'}_{\text{\tiny{$L$}}}(S), \,\,\ \text{for all $S\in \mathcal{S}_\text{{\bsifamily{l}}}^{c}(L)$ and all $L$  object of $\mathbf{Loc}$}. 
\]

 This way $\mathfrak h(\mathbf{Loc})$ inherits a lattice structure from $ \mathcal{S}_\text{{\bsifamily{l}}}^{c}$.

 \begin{prop}
 Every family $( \mathfrak h_{\text{\tiny{$\lambda$}}})_{\text{\tiny{$\lambda\in \Lambda$}}}$ in $\mathfrak h(\mathbf{Loc})$ has a join $\bigvee\limits_{\text{\tiny{$\lambda\in \Lambda $}}} \mathfrak h_{\text{\tiny{$\lambda $}}}$ and a meet $\bigwedge\limits_{\text{\tiny{$\lambda\in \Lambda $}}} \mathfrak h_{\text{\tiny{$\lambda $}}}$ in $Int(\mathbf{Loc})$. The discrete $\mathfrak h$ operator
 \[ \mathfrak h_{\text{\tiny{$D$}}}=({h_{\text{\tiny{$D$}}}}_{\text{\tiny{$L$}}})_{\text{$L\in \mathbf{Loc}$}}\,\,\ \text{with}\,\,\ {h_{\text{\tiny{$D$}}}}_{\text{\tiny{$L$}}}(S)=S\,\,\ \text{for all}\,\ S\in \mathcal{S}_\text{{\bsifamily{l}}}^{c}(L)
 \]
  is the largest element in $\mathfrak h(\mathbf{Loc})$, and the trivial $\mathfrak h$ operator
  \[ 
\mathfrak h_{\text{\tiny{$T$}}}=({h_{\text{\tiny{$T$}}}}_\text{\tiny{$L$}})_{\text{$L\in \mathbf{Loc}$}}\,\,\ \text{with}\,\,\ {h_{\text{\tiny{$T$}}}}_{\text{\tiny{$L$}}}(S)=
  \begin{cases}
\{1\}& \text{for all}\,\ S\in \mathcal{S}_\text{{\bsifamily{l}}}^{c}(L),\,\ S\ne L\\
L&\text {if}\,\ S=L
\end{cases}
 \]
  is the least one.
 \end{prop}

\begin{proof}
For $\Lambda\ne\emptyset$, let $\widehat{\mathfrak h}=\bigvee\limits_{\text{\tiny{$\lambda\in\Lambda $}}}\mathfrak h_{\text{\tiny{$\lambda $}}}$, then 
 \[
 \widehat{h_{\text{\tiny{$L$}}}}=\bigvee\limits_{\text{\tiny{$\lambda\in \Lambda$}}} {h_{\text{\tiny{$\lambda $}}}}_{\text{\tiny{$L$}}},
 \]
 for all  $L$ object of $\mathbf{Loc}$, satisfies
\begin{itemize}
 \item $S\cap  \widehat{h_{\text{\tiny{$L$}}}}(S)\subseteq S$,  because $S\cap {h_{\text{\tiny{$\lambda $}}}}_{\text{\tiny{$S$}}}(L)\subseteq S$, for all $S \in\mathcal{S}_\text{{\bsifamily{l}}}^{c}(L)$ and for all $\lambda \in \Lambda$.
\item If $S\subseteq T$ then $S\cap \widehat{h_{\text{\tiny{$L$}}}}(S)\subseteq T\cap \widehat{h_{\text{\tiny{$L$}}}}(T)$, since $S \cup {h_{\text{\tiny{$\lambda $}}}}_{\text{\tiny{$L$}}}(S)\subseteq  T \cup {h_{\text{\tiny{$\lambda $}}}}_{\text{\tiny{$L$}}}(T)$, for all $S,T \in\mathcal{S}_\text{{\bsifamily{l}}}^{c}(L)$ and for all $\lambda \in \Lambda$.
\item $L\cap  \widehat{h_{\text{\tiny{$L$}}}}(L)= L$,  because $L\cap {h_{\text{\tiny{$\lambda $}}}}_{\text{\tiny{$L$}}}(L)= L$  for all $\lambda \in \Lambda$.
 \end{itemize}
Similarly  $\bigwedge\limits_{\text{\tiny{$\lambda\in \Lambda $}}}\mathfrak h_{\text{\tiny{$\lambda $}}}$,\,\  $ \mathfrak h_{\text{\tiny{$D$}}}$ and $\mathfrak h _{\text{\tiny{$T$}}}$ are $\mathfrak h$ operators.
 \end{proof}
 \begin{coro}\label{h-complete}
  For  every  object $L$ of $\mathbf{Loc}$
  \[
  \mathfrak h(L) = \{h_{\text{\tiny{$L$}}}\mid h_{\text{\tiny{$L$}}}\,\ \text{ is an $\mathfrak h$ operator on}\,\ L\}
  \]
  is a complete lattice.
\end{coro}
 
\subsection{Initial $\mathfrak h$ operators}
Let   $\mathbf{\mathfrak h\text{-}Loc}$ be the category of $\mathfrak h$-spaces. Let  $(M,h_{\text{\tiny{$M$}}})$ be an object of  $\mathbf{\mathfrak h\text{-}Loc}$ and let $L$ be an object of $\mathbf{Loc}$. For each morphism  $f:L\rightarrow M$ in $\mathbf{Loc}$ we define on $L$ the operotor 

\begin{equation} \label{h-initial}
h_{\text{\tiny{$L_{f}$}}}:=f_{-1}\centerdot h_{\text{\tiny{$M$}}}\centerdot f_{*}.
\end{equation}

\begin{prop}\label{h-ini-cont}
The operator (\ref{h-initial}) is an $\mathfrak h$ operator on $L$ for which the morphism $f$ is $\mathfrak h$-continuous.
\end{prop}

\begin{proof}\
\begin{enumerate}
\item[($h_1)$]  $S\cap h_{\text{\tiny{$L_{f}$}}}(S)= f_{-1} \Big[f_{*}[S]\cap h_{\text{\tiny{$M$}}}\big[ f_{*}[S]\big]\Big]\subseteq f_{-1}\big[f_{*}[s]\big]\subseteq S$, 

for all $S\in \mathcal{S}_\text{{\bsifamily{l}}}^c(L)$.

\item[($h_2)$]  $S\subseteq T$ in $\mathcal{S}_\text{{\bsifamily{l}}}^c(L)$, implies $f_{*}[S]\subseteq f_{*}[T]$, then 

$f_{*}[S]\cap h_{\text{\tiny{$M$}}}\big( f_{*}[S]\big)\subseteq f_{*}[T]\cap h_{\text{\tiny{$M$}}}\big( f_{*}[T]\big)$, therefore

$f_{-1}\Big(f_{*}[S]\cap h_{\text{\tiny{$M$}}}\big( f_{*}[S]\big)\Big)\subseteq f_{-1}\Big(f_{*}[T]\cap h_{\text{\tiny{$M$}}}\big( f_{*}[T]\big)\Big)$,
consequently   $S\cap h_{\text{\tiny{$L_{f}$}}}(S)\subseteq T\cap h_{\text{\tiny{$L_{f}$}}}(T)$,  for all $S,T \in\mathcal{S}_\text{{\bsifamily{l}}}^{c}(L)$;
\item[($h_3)$]  $L\cap h_{\text{\tiny{$L_{f}$}}}(L)= f_{-1} \Big[f_{*}[L]\cap h_{\text{\tiny{$M$}}}\big[ f_{*}[L]\big]\Big]=L$. 
\end{enumerate}
\end{proof}

It is clear that  $h_{\text{\tiny{$L_{f}$}}}(L)$ is the coarsest $\mathfrak h$ operator on $L$ for which the morphism $f$ is $\mathfrak h$-continuous; more precisaly
\begin{prop}\label{h-unique}
Let $(L,h_{\text{\tiny{$L$}}})$ and $(M,h_{\text{\tiny{$M$}}})$ be objects of $\mathbf{\mathfrak h\text{-}Loc}$,and let $N$ be an object of  $\mathbf{Loc}$. For each morphism  $g:N\rightarrow L$ in  $\mathbf{Loc}$ and for\linebreak  
$f:(L,h_{\text{\tiny{$L_{f}$}}})\rightarrow (M,h_{\text{\tiny{$N$}}})$ an $\mathfrak h$-continuous morphism, $g$  is $\mathfrak h$-continuous if and only if $f\centerdot g$ is $\mathfrak h$-continuous.
\end{prop}
\begin{proof}
Suppose that $g\centerdot f$ is $I$-continuous, i. e.
$$(f\centerdot g)_{-1}\big(T\cap h_{\text{\tiny{$M$}}}(T)\big)\subseteq T\cap h_{\text{\tiny{$N$}}}\big( (f\centerdot g)_{-1}(T) \big)$$
 for all $T\in T \in\mathcal{S}_\text{{\bsifamily{l}}}^{c}(N)$. Then, for all $S\in T \in\mathcal{S}_\text{{\bsifamily{l}}}^{c}(L)$, we have
\begin{align*}
 g_{-1}\Big(S\cap \big(h_{\text{\tiny{$L_{f}$}}}(S)\big)\Big)&=g_{-1}\Big(f_{-1}\big( f_{*}(S)\cap h_{\text{\tiny{$M$}}}\centerdot f_{*}(S)\big)=(f\centerdot g)_{-1}\big( f_{*}(S)\cap h_{\text{\tiny{$M$}}}( f_{*}(S)) \big)\\
 &\subseteq f\centerdot g)_{-1}(f_{*}(S) \cap  \Big(h_{\text{\tiny{$N$}}}\big( (f\centerdot g)_{-1}(f_{*}(S) ) \big)\Big)\\
 &=(f\centerdot g)_{-1}(f_{*}(S) ) \cap h_{\text{\tiny{$N$}}}\big( g_{-1}\centerdot f_{-1}\centerdot  f_{*} (S)\big)\\
 &\subseteq  g_{-1}(S)\cap h_{\text{\tiny{$N$}}}\big( g_{-1}(S)\big),\\
\end{align*}
i.e.  $g$  is $I$-continuous.
\end{proof}
As a consequence of corollary(\ref{h-complete}), proposition(\ref{h-ini-cont}) and proposition (\ref{h-unique}) (cf. \cite{AHS} or \cite{JM}), we obtain 
 \begin{theorem}
The forgetful functor $U:\mathbf{\mathfrak h\text{-}Loc}\rightarrow \mathbf{Loc}$ is topological, i.e. the concrete category $\big(\mathbf{\mathfrak h\text{-}Loc},\ U\big)$ is topological.
 \end{theorem}


\begin{thebibliography}{10}
\bibitem{AHS}{\sc Jiri Adamek, Horst Herrlich, George Strecker}, {\it Abstract and Concrete Categories}, John Wiley \& Sons, New York, 1990.
\bibitem{Du} {\sc J. Dugundji}, {\it Topology}, Allyn and Bacon, Inc., Boston / London / Sydney / Toronto, 1966.
\bibitem{PJ1}{\sc P. T. Johnstone}, {\it Complemented sublocales and open maps}, Annals of Pure and Applied Logic 137, 2006.
\bibitem{PJ2}{\sc P. T. Johnstone}, {\it Stone spaces}, Cambridge University Press, Cambridge, 1982.
\bibitem{K1}{\sc K. Kuratowski}, {\it Topology, Vol. 1}, Academic Press, New York and London, 1966.
\bibitem{K2}{\sc K. Kuratowski}, {\it Topology, Vol. 2}, Academic Press, New York and London, 1968.
\bibitem{LO} {\sc J. Luna-Torres, C. Ochoa}, {\it Interior operators and topological categories}, Adv. Appl. Math. Sci., 10, 2011.
\bibitem{SM}{\sc S. MacLane}, {\it Categories for the Working Mathematician}, Springer-Verlag, New York / Heidelberg / Berlin,1971.
\bibitem{MM} {\sc S. MacLane} and {\sc I. Moerdijk}, {\it Sheaves in Geometry and Logic,{ \scriptsize A first introduction to Topos theory}}, Springer-Verlag, New York / Heidelberg / Berlin,1992.
\bibitem{PP} {\sc J. Picado, A. Pultr}, {\it Frames and Locales: Topology Without Points}, Frontiers in Mathematics, vol. 28, Springer, Basel, 2012.
\bibitem{ALS} {\sc A. L. Suarez}, {\it Revisiting the relation between subspaces and sublocales}, arXiv:2010.05284v1 [math.FA] 11 Oct 2020.
\bibitem{MSM} {\sc M. Suarez M.}. {\it La funci\'on topol\'ogica h}, Universidad Pedag\'ogica y Tecnol\'ogica de Colombia, Tunja, 1988.
\end{thebibliography}
\end{document}